\documentclass{amsart}

\usepackage[utf8]{inputenc}
\usepackage[T1]{fontenc}
\usepackage{amsmath}
\usepackage{amssymb}
\usepackage{amsthm}
\usepackage{graphicx}
\usepackage{subfig}
\usepackage{xcolor}

\newcommand{\xcrit}{x^{(c)}}
\newcommand{\alphacrit}{\alpha^{(c)}}
\newcommand{\e}{\mathrm{e}}
\renewcommand{\i}{\mathrm{i}}
\newcommand{\real}[1]{\mathrm{Re}\{#1\}}
\newcommand{\imag}[1]{\mathrm{Im}\{#1\}}
\newcommand{\expsin}{\mathrm{s}}
\newcommand{\expcos}{\mathrm{c}}
\newcommand{\R}{\mathbb{R}}

\newcommand{\AdditionalInfo}[1]{ }  

\newtheorem{proposition}{Proposition}
\newtheorem{corollary}{Corollary}

\begin{document}
\title[Robust optimization of delay differential equations]{A method for the optimization of nonlinear systems with delays that guarantees stability and robustness}
\author{Jonas Otten}%\address{Automatic Control and Systems Theory, Ruhr-Universit\"at Bochum, Germany}
\author{Martin Mönnigmann}
\date{dates will be set by the publisher}
\date{\today{}}

\begin{abstract}  
We present a method for the steady state optimization of nonlinear delay differential equations. The method ensures stability and robustness, where a system is called robust if it remains stable despite uncertain parameters. Essentially, we ensure stability of all steady states of the nonlinear system on the steady state manifold that results from the variation of the uncertain parameters. The uncertain parameters are characterized by finite intervals, which may be interpreted as error bars and therefore are of immediate practical relevance. Stability despite uncertain parameters can be guaranteed by enforcing a lower bound on the distance of the optimal steady state to submanifolds of saddle-node and Hopf bifurcations on the steady state manifold. We derive constraints that ensure this distance. The proposed method differs from previous ones in that stability and robustness are guaranteed with constraints instead of with the cost function. Because the cost function is not required to enforce stability and robustness, it can be used to state economic or similar goals, which is natural in applications. We illustrate the proposed method by optimizing a laser diode. The optimization finds a steady state of maximum intensity while guaranteeing asymptotic or exponential stability despite uncertain model parameters. 
\end{abstract}
\maketitle

\section{Introduction}\label{sec:Intro}
Delays can significantly influence the behavior of a dynamical system.  
In fact, delays can have both a stabilizing and a destabilizing effect~\cite{Sipahi2011}. Technically, it is more difficult to analyze a delayed systems, 
because even a single delay leads to an infinite dimensional dynamical system~\cite{Fridman2010}.

Delays are present in various kinds of dynamical systems. They are, for example, part of traffic models where they model reaction times of human drivers~\cite{Helbing2001}, they influence population dynamics in multi-stage populations~\cite{Aiello1992, Adimy2010}, they cause limit cycles in milling and turning processes~\cite{Seguy2010, Insperger2008}, and they can destabilize supply chains~\cite{Sipahi2010}.
The delay present in external cavity lasers causes very interesting dynamical behaviour \cite{Lang1980,Mork1992,Haegeman2002}. External cavity lasers can exhibit even quasi-chaotic dynamics. The relevant time constants are typically dominated by photon lifetimes of around 1ps~\cite[p.\,232]{Agrawal1993}. As a consequence, it is challenging to implement feedback control, and thus designs and operating modes that are intrinsically stable and robust are of interest. 

Parameter variations that lead to a significant change in the behavior of semiconductor laser dynamics can be associated with the occurrence of bifurcations. While a Hopf bifurcation is desirable at the initial lasing threshold, further bifurcations might render the single-wavelength operation unstable.
Since a loss of stability involves a real or conjugate pair of eigenvalues crossing the imaginary axis into the right half complex plane, it is obvious to search for laser parameters that result in modes of operations with eigenvalues with sufficiently negative real parts. 
Vanbiervliet et al.~\cite{Vanbiervliet2008} use nonlinear programming methods to find laser diode parameters that result in the minimal (i.e., maximally negative) leading eigenvalue.

The method proposed in the present paper is also based on nonlinear programming and on characterizing stability and robustness with the distance of the leading eigenvalue to the imaginary axis. Our method is fundamentally different from the one by Vanbiervliet et al.~\cite{Vanbiervliet2008}, however. 
We do not use the cost function to enforce the desired eigenvalues, but 
%We present an alternative approach to achieving stability. The proposed method is also based on mathematical programming. Instead of using the cost function for enforcing stability requirements as in~\cite{Vanbiervliet2008}, 
we impose additional constraints that establish robust stability (see Section~\ref{sec:NVMethod}). 
Moreover, we use the distance to the closest bifurcation in the parameter space as a measure for robustness (see the next paragraph for a brief explanation). 
Since the cost function is not required to enforce stability, other optimization goals can be stated; in our case this is the laser intensity. 
Essentially, we can systematically determine the mode of the highest laser intensity that is sufficiently robust by optimizing with respect to the laser intensity and enforcing a lower bound on the distance to all stability boundaries. Note that this is different from determining the most robust mode by pushing all eigenvalues as far into the left half of the complex plane as possible.

The distance of a candidate steady state of operation to the closest bifucation manifold can serve as a measure for robustness. Dobson showed this distance occurs along a particular normal vector to the bifurcation manifold~\cite{Dobson1993}. Mönnigmann and Marquardt incorporated these normal vectors to state robustness in economic optimization problems~\cite{Monnigmann2002}. Their normal vector approach has been extended in different contexts, for example, to achieve specified transient behavior~\cite{Gerhard2008}, to the stability of periodically operated systems \cite{Kastsian2010} and, more recently, to delayed systems~\cite{Kastsian2013, Otten2016b, Otten2016a}. The Lang-Kobayashi type laser diode models, which serve as examples in the present paper, belong to the latter problem class.

After introducing the system class in the remainder of Section~\ref{sec:Intro}, we present a sample system, a laser model, and discuss the optimization task in Section~\ref{sec:LaserProblemOutline}. 
In Section~\ref{sec:NVMethod}, we propose normal vector constraints as an instrument for achieving robust stability and derive constraints for robust asymptotic and robust exponential stability. We apply these normal vector constraints to the laser optimization in Section~\ref{sec:LaserNVOptim}. Conclusions and an outlook are given in Section~\ref{sec:Conclusion}. 

\subsection{System Class and notation}
We consider the class of delay differential equations (DDEs) with multiple uncertain and state dependent delays
\begin{equation}\label{eq:sys}
  \dot x(t)=f(x(t),x(t-\tau_1),\dots,x(t-\tau_m),\alpha)
\end{equation}
with the state vector $x \in \mathbb{R}^{n_x}$, 
%, a vector of algebraic variables $\Omega \in \mathbb{R}^{n_\Omega}$, 
uncertain parameters $\alpha \in \mathbb{R}^{n_\alpha}$, where $f$ is smooth and maps from $\mathbb{R}^{n_x(m+1)}\times \mathbb{R}^{n_\alpha}$, or an open subset thereof, into $\mathbb{R}^{n_x}$.
We assume there exist $m$ state and parameter dependent delays
\begin{equation}\label{eq:delay}
\tau_i=\tau_i(x(t),\alpha),\, i= 1, \dots, m.
\end{equation}

We refer to $\lambda \in \mathbb{C}$ as an eigenvalue at a steady state $ x $ of \eqref{eq:sys}, if
\begin{equation}\label{eq:det}
\mathrm{det}\left(\lambda I - A_0 -\sum_{i=1}^mA_i\exp(-\lambda \tau_i)\right)=0\,,
\end{equation}
where $A_0$ and $A_i$ are the Jacobians of the right hand side of \eqref{eq:sys} with respect to $x(t)$ and $x(t-\tau_i)$ respectively (see, e.g., Engelborghs and Roose \cite{Engelborghs1999}).
The delay $\tau_0=0$ is merely introduced to simplify the notation. This allows us to replace $x(t)$ by $x(t-\tau_0)$ and refer to $x(t-\tau_i)$, $i= 0, \dots, m$ instead of referring to $x(t)$ and $x(t-t_i)$, $i= 1, \dots, m$ separately.  
Furthermore, we introduce the abbreviations
\begin{subequations}
\begin{align}
\expsin(\sigma,\omega,\tau)&=\exp(-\sigma\tau)\sin(\omega\tau)\\
\expcos(\sigma,\omega,\tau)&=\exp(-\sigma\tau)\cos(\omega\tau)\,.
\end{align}
\end{subequations}
A solution of a set of nonlinear equations is called regular if the Jacobian of the nonlinear equations evaluated at this solution has full rank. 

\section{Laser diode model and problem outline}\label{sec:LaserProblemOutline}
We use a laser diode model to motivate the use of the normal vector method to be introduced in Section~\ref{sec:NVMethod}. 
The model was proposed by Verheyden et al.~\cite{Verheyden2004}. 
It is of the Lang-Kobayashi type and models the laser dynamics in the presence of an external cavity. 
The external cavity feeds a small fraction of the emitted light back into the semiconductor with a delay, thus necessitating a delay differential equation (DDE). 
Since the semiconductor is represented by a diffusion equation, the overall system model requires coupling a delay differential to a partial differential equation. We use a method-of-lines approximation of the partial differential equation as proposed in~\cite{Verheyden2004}.
A technical complication arises because of the rotational symmetry of the steady state solutions of the Lang-Kobayashi model~\cite{Haegeman2002,Krauskopf2000}. 

\subsection{Dynamical laser diode model}\label{sec:Model}
\begin{subequations}\label{eq:ODECarrierModelRot}
The electrical field $E(t)\in\mathbb{C}$ of the laser diode can be modeled with 
\begin{align}
\frac{\mathrm{d}A(t)}{\mathrm{d}t}=&-\mathrm{i} \Omega A(t) + (1-\mathrm{i}\alpha_{LW})A(t)\zeta(t)\nonumber\\
&+\eta\,\e^{-\mathrm{i}(\phi-\Omega \tau)} A(t-\tau)\label{eq:discElectricalField}\,,
\end{align}
where $E(t)=A(t)\exp(\mathrm{i}\Omega t)$. 
The variable $A(t)\in\mathbb{C}$ measures the electrical field in rotating coordinates and captures harmonic oscillations in $E(t)$ (i.e., a single wavelength operation) as a steady state of the frequency $\Omega$.  
Model parameters are given in Table~\ref{tab:Params}.
\begin{table}
	\begin{tabular}{llrrr}
		\hline
		symbol &parameter& uncertainty & nom. value\\
		\hline
		$j$& pump current &  $\pm 0.01 $& $j^{(0)}$, s.t. opt.\\
		$\eta$ & feedback strength  &  $\pm 0.0001 $ & $\eta^{(0)}$, s.t. opt.\\
		$\alpha_{LW}$& linewidth enhancement factor&  $\pm 1 $& as in \cite{Verheyden2004}\\
		$\phi$ & feedback phase&$\pm0.1\pi$ & as in \cite{Verheyden2004}\\
		$\tau$&  feedback delay&  $\pm 10 $ & $\tau^{(0)}$, s.t. opt. \\
		$T$ & carrier to photon lifetime  &  $\pm 10 $& as in \cite{Verheyden2004}\\
		$d$& diffusion constant &   $\pm 0.001 $&as in \cite{Verheyden2004}\\
		\hline
	\end{tabular}
	\caption{System parameters. Nominal values are taken from Verheyden et al.~\cite{Verheyden2004}. The abbreviations 's.t.' and 'opt.' are short for 'subject to' and 'optimization', respectively.}\label{tab:Params} 
\end{table}

The energy powering the laser is supplied through the semiconductor. The partial differential equation modeling the carrier diffusion can be spatially discretized with a method-of-lines approach. According to Verwer and Sanz-Serna~\cite{Verwer1984}, the discretization converges to the underlying PDE solution for a sufficient number of discrete points. Let $k = 1,2,...,K$  refer to the points resulting from the spatial discretization.
Then
\begin{align}
T\frac{\mathrm{d} N_k(t)}{\mathrm{d}t}=&d \frac{ N_{k-1}(t) - 2  N_k(t) + N_{k+1}(t)}{h^2}-N_k(t)\nonumber\\
& +P_k- F_k[1+2\,N_k(t)]|A(t)|^2\label{eq:discCarrier}
\end{align}
result for the discretized carrier density $N_k(t)$ 
in the interior ($k=2,3,\dots, K-1$),
where the constant $h=\frac{0.5}{K-1}$ describes the distance of the discrete spatial points.
Assuming Neumann boundary conditions, the remaining two equations for $k=1$ and $k=K$ are
\begin{align}
T\frac{\mathrm{d} N_1(t)}{\mathrm{d}t}=&d \frac{ N_2(t) -  N_1(t)}{h^2}-N_1(t)\nonumber\\
& +P_1- F_1[1+2\,N_1(t)]|A(t)|^2\label{eq:discBounds1}
\end{align}
and
\begin{align}
T\frac{\mathrm{d} N_K(t)}{\mathrm{d}t}=&d \frac{ N_{K-1}(t) -  N_K(t)}{h^2}-N_K(t)\nonumber\\
& +P_K- F_K[1+2\,N_K(t)]|A(t)|^2\label{eq:discBoundsl}\,.
\end{align}

The carrier diffusion is driven by an anisotropic pump current. Its discretized spatial distribution is
\begin{equation}
P_k=\left\{
\begin{array}{ll}
1.075j& \text{ for } |k\cdot h| < 0.2 \\
-0.8& \text { else }
\end{array}
\right\},~k=1,\dots,K\,.
\end{equation}
The modal gain
\begin{equation}
\zeta(t)= 2h\sqrt{\frac{400}{\pi}}\sum_{k=1}^{K} F_k N_k(t)
\end{equation}
depends on the spatial distribution of the electrical field 
\begin{equation}
F_k=  \exp(-400(k\cdot h)^2), k=1,\dots,K\,.
\end{equation}
 
When stated in rotating coordinates as in~\eqref{eq:discElectricalField}, the rotational symmetry of the Lang-Kobayashi model (see~\cite{Krauskopf2000}) is evident from the fact that $A(t)\exp(i\theta)$ is a steady state solution to~\eqref{eq:discElectricalField} for any $\theta\in\R$, if $A(t)\in\mathbb{C}$ is a steady state solution. 
We follow Verheyden et al.~\cite{Verheyden2004} in removing this indeterminacy with the additional condition
\begin{equation}\label{eq:phaseCondLaser}
\real{A(t)} -\imag{A(t)}=0. 
\end{equation}
\end{subequations}
It is known that the stability properties of the steady states are not affected by this additional algebraic equation (see~\cite{Haegeman2002}, Section~II). 
Consequently, steady states of~\eqref{eq:sys} correspond to solutions $0= f(x, x, \alpha)$ of the $n_x= K+3$ equations
\begin{align*}
  f(x(t), x(t-\tau), \alpha)= \begin{bmatrix} \eqref{eq:discElectricalField} \\ \vdots \\ \eqref{eq:phaseCondLaser}\end{bmatrix}
\end{align*}
in the $n_x$ variables 
$x= [\real{A}, \imag{A}, N_1, \dots, N_K, \Omega]^\prime$.
%The resulting system has $n_x=2+K$ state variables $x = [\mathrm{Re}\{A\},\mathrm{Im}\{A\},N_1,...,N_K]^\prime$,  
%the algebraic variable $\Omega$ ($n_\Omega = 1$) that results from the representation in the rotating coordinate frame and 
The model has $n_\alpha = 7$ parameters $\alpha = [j, \eta, \alpha_{LW}, \phi, \tau, T, d]^\prime$. We consider all parameters to be uncertain in the sense that they are known up to an error bar only. The uncertainties or error bars are given in Table~\ref{tab:Params}. 

\subsection{Optimization problem outline}\label{ssec:Optimization}
It is the objective of the following optimization to find the parameter configuration for which the laser intensity becomes maximal. A high laser intensity is important for applications relying on the high power density of lasers such as laser cutting \cite{Mas2003} or laser cladding \cite{Thompson2015}. We can state the task of maximizing the stationary laser intensity as
\begin{subequations}\label{eq:optimProb}
\begin{align}
\min_{x,j, \eta, \tau}~-|A|^2&\label{eq:optimProbCost}\\
\text{s.\,t.}\quad0 &= f(x,x,\alpha)\label{eq:optimProbStSt}
\end{align}
where we assume the optimal values of parameters pump current $j$, feedback strength $\eta$ and feedback delay $\tau$  can be determined by the optimization. Note that these parameters are chosen by the optimization and simultaneously assumed to be uncertain, which reflects that the optimal parameters can only be fixed up to an error bar, just as any fixed parameter can only be determined up to an error bar. 
We choose the parameters $j$, $\eta$ and $\tau$ for an optimization, because they can be changed after the laser has been manufactured. We assume they can be chosen freely as long as lower and upper bounds are met, which read 
\begin{align}
	0.02 \leq & j \leq 1.0265  \label{eq:NonrobustConstraintj} \\
	0.0001 \leq &\eta \leq 0.02  \label{eq:NonrobustConstrainteta} \\
	220 \leq &\tau \leq 2000\,. \label{eq:NonrobustConstrainttau}
\end{align}
\end{subequations}
The bounds on the feedback strength $\eta$ model limitations in the coating process for the semi-reflective mirror at the end of the external cavity. Bounds on $\tau$ are necessary to account for geometric limitations of the external cavity length. The pump current bounds represent a limited range of a current source.
%Both the adjustable parameters $j$, $\eta$, $\tau$ and the parameters fixed by the manufacturing process $\alpha_{LW}$, $\phi$, $T$, $d$ are subject to  uncertainty. The assumed uncertainties are given in Table~\ref{tab:Params}.

\subsection{Limitations of naive steady state optimizations of dynamical systems}
\subsubsection{Optimization without stability constraints result in unstable operation}\label{sssec:UnStabResults}
The numerical solution of the the optimization problem \eqref{eq:optimProb} for $K=31$ results in the optimal parameters  $j=1.0265$, $\eta = 0.02$ and $\tau=220$. These parameters correspond to a stationary intensity of $|A|^2 = 6.6622$. While optimal, this steady state is unstable, which is evident from Figure~\ref{fig:simUnStable}. 
\begin{figure}[hbtp]
	\center
	\includegraphics[]{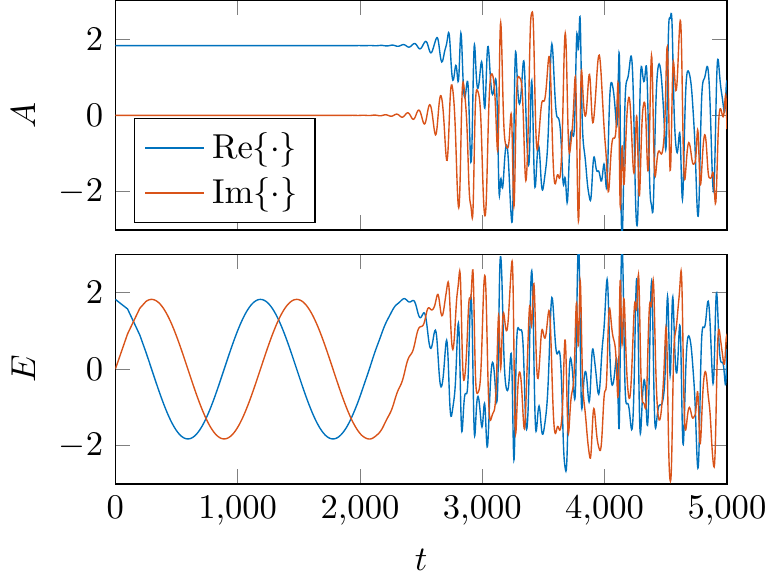}
	\caption{Simulation of laser diode \eqref{eq:ODECarrierModelRot} for $K=31$ at the parameters that result for the optimization~\eqref{eq:optimProb} without stability and robustness constraints. The initial steady state is unstable and the system evolves quasi-chaotically after a small disturbances due to numerical inaccuracy.  
	Results are shown both in rotating coordinates $A$ and fixed coordinates $E$.
	\label{fig:simUnStable}}
\end{figure}

This result illustrates an optimization such as~\eqref{eq:optimProb} that ignores the stability properties of the dynamical system is likely to fail,
because 
an optimal but unstable, or more generally, an optimal but not sufficiently robust mode of operation is useless.  
Arguably, the intensity could first be optimized, and the resulting steady state could be stabilized by adding feedback control a posteriori. 
For systems with very fast intrinsic time scales this is not a viable option, however. Moreover, even if such an a posteriori stabilization was possible, it would raise the question whether there exists a competitive intrinsically stable and robust mode of operation that does not require adding a controller device to the laser diode system.
The remainder of the paper is devoted to an optimization method that systematically takes stability and robustness into account. 
%This is instability is very important, because laser dynamics are too fast for most controller design approaches. A stabilizing feedback controller can therefore not be found methodically. Without stabilization, the optimization result is useless, because the system cannot be operated at an unstable steady state. It is therefore necessary to guarantee stability already during the optimization. The approach for this is sketched in the next section.

\section{ Normal vector method }\label{sec:NVMethod}
%\begin{table}
%  \begin{tabular}{p{2cm}|p{2cm}|p{1cm}|p{1cm}|p{1cm}|p{4cm}}
%   journal/conf & title & mult $\tau_i$ & $\tau_i(x)$ & $\tau_i(x, \alpha)$ & comments 
%   \\\hline
%    IFAC Workshop TDS~\cite{Otten2016a} 
%    & Robust steady state  opt with state dep delays
%    & y
%    & y
%    & n
%    & fold, gen fold, 
%    (not Hopf, not gen Hopf),
%    population balance
%  \end{tabular}
%\caption{Internal use only: Summary of normal vector systems and their applications in previous publications.}
%\end{table}
We give an informal introduction to the simple geometric idea of the proposed method first. Subsequent sections then use more precise notions from bifurcation theory. 
\subsection{ Normal vectors to enforce stability } \label{ssec:NVMotivation}

A steady state is asymptotically stable, if the eigenvalues defined by \eqref{eq:det} are all in the left half-plane, i.e., 
\begin{equation}\label{eq:eigAsymptStabCond}
\mathrm{Re}\{\lambda_i\} < 0 \text{ for all } \lambda_i \,.
\end{equation} 
The system loses stability if a variation of the parameters $\alpha_i$ causes one or more eigenvalues to cross from the left into the right half plane. 
Assuming no higher codimension bifurcation points appear, the critical points constitute a manifold that locally separates the parameter space into a stable and an unstable part \cite{Kuznetsov1998} (see Figure~\ref{fig:NVSketch}). 
We sometimes prefer to enforce exponential stability with a prescribed decay rate $\sigma< 0$ and replace~\eqref{eq:eigAsymptStabCond} by
\begin{equation}\label{eq:eigExpStabCond}
\mathrm{Re}\{\lambda_i\} <\sigma< 0~\text{ for all }\lambda_i \,.
\end{equation}
In this case the existence of 
real eigenvalue or complex conjugate pair of eigenvalues with real part $\sigma$ defines the separating manifold. We call this manifold \emph{critical manifold} in both cases, \eqref{eq:eigAsymptStabCond} and~\eqref{eq:eigExpStabCond}, by a slight abuse of terminology. 
%It is in fact easier to characterize critical manifolds that belong to the violation of~\eqref{eq:eigExpStabCond}, because they do not involve bifurcating solutions. 

The closest distance of a point $\alpha^{(0)}$ in the parameter space to the critical manifold can serve as a measure for robustness
(see Figure~\ref{fig:NVSketch}).  
The closest distance occurs along a direction that is normal to the critical manifold~\cite{Dobson1993}. 
This direction and the closest critical point are labeled $r$ and $\alpha^{(c)}$, respectively, in Figure~\ref{fig:NVSketch}. 
It is evident from the figure 
that the distance $d$ between $\alpha^{(0)}$ and $\alpha^{(c)}$ respects
\begin{equation}
\alpha^{(0)}-\alpha^{(c)} + d\frac{r}{||r||} = 0\,. 
\end{equation}
Stability can be enforced with the constraint $d>0$ in an optimization. 
\begin{figure}[hbtp]
	\center
	\includegraphics[]{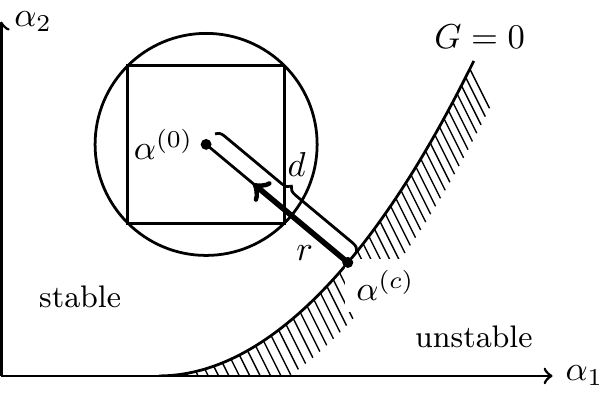}
	\caption{\label{fig:NVSketch} 
	Basic idea of the normal vector method. 
	The line passing through the candidate point of operation $\alpha^{(0)}$ and the closest critical point $\alphacrit$ is normal to the critical manifold.}
\end{figure}
Robustness can also be enforced easily as follows. We assume 
the uncertainty is quantified by error bars around the nominal values $\alpha^{(0)}$
\begin{equation}
	\alpha_i\in[\alpha_i^{(0)} -\Delta \alpha_i,\alpha_i^{(0)} +\Delta \alpha_i]\,,\label{eq:uncertaintyHyperrectangle}
\end{equation}
where $\Delta \alpha_i$ measures the uncertainty of the respective $\alpha_i$. Rescaling the parameter space by these
uncertainty intervals\footnote{Alternatively, the $\Delta\alpha_i$ can be used to define a metric.}
transforms the hyperrectangular uncertainty region \eqref{eq:uncertaintyHyperrectangle} into a hypercube,
\begin{equation}
	\frac{\alpha_i}{\Delta \alpha_i}\in\left[\frac{\alpha_i^{(0)}}{\Delta \alpha_i} -1,\frac{\alpha_i^{(0)}}{\Delta \alpha_i} +1\right]\label{eq:uncertaintyHypercube}\,.
\end{equation}
Enforcing a parametric distance of $d>\sqrt{n_\alpha}$ leads to a robustly stable steady state, because the uncertainty region \eqref{eq:uncertaintyHypercube} is enclosed in a hypersphere of radius $\sqrt{n_\alpha}$ that cannot cross the stability boundary~\cite{Monnigmann2002}. The geometry of the hypercubic uncertainty region, the enclosing hypersphere and the stability boundary are  illustrated in Figure~\ref{fig:NVSketch} for $n_\alpha=2$.

Applied bifurcation theory states systems of equations that characterize critical manifolds for the case~\eqref{eq:eigAsymptStabCond} in the form of \textit{augmented systems} (see, e.g., \cite{Kuznetsov1998}). 
%The augmented systems required here are stated in Sect.~\ref{ssec:NVSysDeriv}. \textcolor{red}{(Stimmt die Ankuendigung auch?)}
They have the form   
\begin{equation}\label{eq:generalAugSys}
0=G(x^{(c)},\alpha^{(c)},u^{(c)})\,,
\end{equation}
where $x^{(c)}$ is a steady state on the stability boundary for the critical parameters $\alpha^{(c)}$. 
The vector $u^{(c)}$ collects various auxiliary variables such as eigenvectors that belong to the critical eigenvalues. 
We also use the term \textit{augmented systems} for systems of equations used to characterize critical manifolds for~\eqref{eq:eigExpStabCond}. 
Section~\ref{ssec:NVSysDeriv} explains how to extend these systems for the calculation of the normal vectors $r$ required for the robustness constraints introduced with Figure~\ref{fig:NVSketch}. 
These extensions, which can be found with a procedure stated in~\cite{Monnigmann2002}, result in systems of nonlinear equations of the form 
%
%The normal vector defining system of nonlinear equations can be derived as proposed in \cite{Monnigmann2002} by calculating the normal space of the stability boundary \eqref{eq:generalAugSys}, which is spanned by its Jacobian's rows. The result is a system of nonlinear equations 
\begin{equation}\label{eq:generalNVSys}
0=H(x^{(c)},\alpha^{(c)},u^{(c)},\kappa,r)\,,
\end{equation}
where $r\in\mathbb{R}^{n_\alpha}$ is the normal vector at the critical point $x^{(c)}$, $\alpha^{(c)}$ and $u^{(c)}$ from \eqref{eq:generalAugSys}. The vector $\kappa$ collects additional auxiliary variables. Just as in~\eqref{eq:generalAugSys}, the number of equations in $0= G$ and $0= H$ depends on the type of critical point. 
This particular types of interest here are treated in the next section. 

%\subsection{Stability boundaries and bifurcation theory}
%
%The proposed approach belongs to the class of normal vector methods. It utilizes that augmented systems of codim-1 bifurcations can also be used to detect the loss of stability. According to those results, parameter space is divided into different regions. In each region, the dynamical behavior is similar or, in technical terms, topologically equivalent~\cite[Ch.\,2]{Kuznetsov1998}. The manifolds separating different regions of topologically equivalent dynamical behavior are bifurcation manifolds~\cite[Def.\,2.11]{Kuznetsov1998}. 
%
%While augmented systems pose only necessary conditions for the occurrence of a bifurcation, they also capture if there is an eigenvalue on the imaginary axis, e.\,g. because the steady state is becoming unstable at that point.
%
%The stability condition which applies to the optimal steady state can therefore be replaced by a condition which prohibits solvability of a bifurcation's augmented system. This link between bifurcation theory and stability allows us to study the properties of stability boundaries using tools from bifurcation analysis.

\subsection{Derivation of normal vector systems}\label{ssec:NVSysDeriv}
If a steady state belongs to a critical manifold for stability, there exists a zero eigenvalue or a complex conjugate pair on the imaginary axis. Critical manifolds for stability can therefore be characterized with the augmented systems for fold and Hopf bifurcations. 
Critical manifolds of steady states with a prescribed decay rate $\sigma< 0$ can be characterized correspondingly, i.e., by augmented systems that state the existence of a real eigenvalue $\sigma$ or a complex conjugate pair with real part $\sigma$. Because these  systems obviously resemble those for fold and Hopf bifurcations closely, we refer to them as augmented systems for modified fold and modified Hopf points, respectively.

\subsubsection{Normal vectors system for fold and modified fold manifolds}
Assume $\xcrit$ is a steady state for parameter values $\alpha^{(c)}$. 
If $\sigma$ is real and the leading eigenvalue at the steady state $\xcrit$, then there exists a $ w \in \mathbb{R}^n$ such that $\xcrit$, $\alpha^{(c)}$, $a$ and $b$ obey the following equations, which are a simple extension of the standard augmented system for saddle-node bifurcations of delay differential equations (see, e.g., \cite{Engelborghs2002}):
\begin{subequations}\label{eq:ModfoldMani}
  \begin{align}
    f(\xcrit,...,\xcrit,\alphacrit)&=0
    \label{eq:GeneralizedFoldNVSystem-SteadyState}
    \\
    \sigma w-A_0\, w-\sum_{i=1}^m A_i \exp(-\sigma\tau_i)\,w&=0
    \label{eq:GeneralizedFoldNVSystem-CriticalEigenvalue}
    \\
    w^\prime w-1&=0
    \label{eq:GeneralizedFoldNVSystem-Regularization}
      \end{align}
\end{subequations}
  Equations~\eqref{eq:GeneralizedFoldNVSystem-SteadyState} holds, because
  $\xcrit$ is a steady state by assumption. 
  Since $\sigma$ is a real eigenvalue by assumption, \eqref{eq:det} holds. Consequently, there exists an eigenvector $w \in \R^n$ that satisfies~\eqref{eq:GeneralizedFoldNVSystem-CriticalEigenvalue}. 
%  \begin{equation*}
%  \sigma  w - A_0 w -\sum_{i=1}^m A_i w \exp(-\sigma \tau_i)=0\,.
%  \end{equation*}
  Equation \eqref{eq:GeneralizedFoldNVSystem-Regularization} normalizes this eigenvector to unit length. 

The following proposition extends a result stated in~\cite{Otten2016a} to the case with delays that depend on uncertain parameters $\alpha$. 
\begin{proposition}[normal vector to manifold of modified fold points]\label{prop:ModFoldNV}
Let $\sigma\leq0$ be arbitrary but fixed. 
Assume that  $( \xcrit, \alpha^{(c)}, w)$ is a regular solution to \eqref{eq:ModfoldMani} in the $2 n_x +1$ variables  $ \xcrit, w$ and one of the elements of $\alpha^{(c)}$. Then $r$ that obeys the following equations is normal to the manifold of modified fold points at this solution:
\begin{subequations}
\begin{align}
	\text{Equations~\eqref{eq:ModfoldMani}} \\
    \begin{bmatrix}
		\nabla_{\xcrit}f^\prime
		& B_{12} 
		& 0 
		\\
		0
		& B_{22} 
		& 2w 
    \end{bmatrix}\kappa
      & = 0
      \label{eq:GeneralizedFoldNVSystem-NVSpan}
      \\
      \begin{bmatrix}
	    \nabla_{\alpha}f^\prime
	    & B_{32}
	    &0
	    \\
      \end{bmatrix}\kappa
      -r
      & = 0 
      \label{eq:GeneralizedFoldNVSystem-NVProjection}
      \\
      r^\prime r- 1&=0
      \label{eq:GeneralizedFoldNVSystem-Normalization}\,  ,
  \end{align}
\end{subequations}
where
\begin{subequations}
	\begin{align}
	    B_{12}=&\sum_{k=0}^m \exp(-\sigma \tau_k) \Big(
		    \sigma\left[\nabla_{\xcrit}\tau_k\right]w^\prime A_k ^\prime 
		    -\left[\nabla_{\xcrit}w^\prime A_k^\prime \right]
		\Big)
		\label{eq:B12SaddleNode}
	    \\
	    B_{22}=& \sigma I-\sum_{k=0}^m \exp(-\sigma \tau_k) A_k^\prime  
	    \label{eq:B22SaddleNode}
	    \\ 
	    B_{32}=
	    &\sum_{k=0}^m \exp(-\sigma \tau_k)\Big(
	    		\sigma\left[\nabla_{\alpha^{\text{(c)}}}\tau_k\right]w^\prime A_k^\prime
	    		-\left[\nabla_{\alpha^\text{(c)}}w^\prime A_k^\prime \right]
	    	\Big)\,.
	    	\label{eq:B32SaddleNode}
	\end{align}
\end{subequations}
\end{proposition}

\begin{proof}
Consider \eqref{eq:ModfoldMani} as $2n + 1$ equations in the $2n+n_\alpha$ variables $ \xcrit$, $w$ and $ \alpha^{(c)}$.  
These equations  define an $(n_\alpha-1)$-dimensional manifold of modified fold points in the neighborhood of the known regular solution. 
Evaluated at any point on this manifold, 
the rows of the Jacobian of \eqref{eq:ModfoldMani} with respect to $\xcrit$, $w$ and $\alpha^{(c)}$ span the normal space to the manifold \cite{Monnigmann2002}\cite[Chapter\,3]{Strang2006}.
Since we would like to span the normal space with column vectors, it is more convenient to work with the transposed Jacobian, which we denote $B$. 
Formally, $B$ can be stated as the outer product
\begin{align*}
 B
=&\begin{bmatrix}
	\nabla_{ \xcrit}\\
	\nabla_{ w }\\
	\nabla_{ \alpha^{(c)}}\\
\end{bmatrix}
\begin{bmatrix}
	 f( \xcrit,..., \xcrit, \alpha^{(c)})
	 \\
	 \sigma w  -\sum_{k=0}^m \exp(-\sigma \tau_k) A_k w\\
	 w ^\prime w   -1
\end{bmatrix}^\prime\,
\end{align*} 
which results in the block matrix
\begin{align}\label{eq:BblockStructure}
	 B
	&= \begin{bmatrix}
	  \nabla_{x^{(c)}} f^\prime & B_{12} & 0  \\
	  0 & B_{22} & 2w  \\
	  \nabla_{\alpha^{(c)}} f^\prime & B_{32} & 0  	  
	  \end{bmatrix},
\end{align}
where the block columns contain $n$, $n$ and $1$ columns and the block rows contain $n$, $n$ and $n_\alpha$ rows, respectively. 
The elements in the second column of~\eqref{eq:BblockStructure} remain to be determined. 
The block $B_{22}$ 
is the Jacobian of $\sigma w-A_0^\prime\, w-\sum_{i=1}^m A_i^\prime \exp(-\sigma\tau_i)\,w$ with respect to $w$, which results in~\eqref{eq:B22SaddleNode}. 
The block matrices $B_{12}$ and $B_{32}$ in  \eqref{eq:BblockStructure} require the application of the chain and product rules, since 
$\tau_k$ is, as described in \eqref{eq:delay}, a function of $ \xcrit$ and $ \alpha^{(c)}$. 
Therefore, 
\begin{align*}
	 B_{12}=&\nabla_{ \xcrit} \left( \sigma w-\sum_{k=0}^m \exp(-\sigma\tau_k) A_k w \right)^\prime
	 =
	 -\sum_{k=0}^m \Big(
	 	\left[\nabla_{\xcrit} \exp(-\sigma\tau_k)\right] w^\prime A_k^\prime
	 	+\exp(-\sigma \tau_k) \nabla_{\xcrit} w^\prime A_k^\prime
	 \Big)
\end{align*}
which yields~\eqref{eq:B12SaddleNode} and $B_{32}$ as stated in~\eqref{eq:B32SaddleNode} results when $\nabla_{\xcrit}$ is replaced by $\nabla_{\alpha^{(c)}}$. 
%\AdditionalInfo{
%	It remains to state expressions for the
%	%In those expressions we had to calculate the derivative of 
%	vector matrix products such as $\nabla_{ \xcrit} (w ^\prime A_i^\prime)$. In order to do so we have to switch to components. 
%	The matrix $A_i$ contains 
%	\begin{equation*}
%	(A_i)_{\rho,\nu}=\frac{\partial f_\rho}{\partial  \xcrit_\nu(t-\tau_i)}\,
%	\end{equation*}
%	in its $\rho$-th row and $\nu$-th column. 
%	Multiplying it with a vector $ w $ from the right results in $A_i w$ with components
%	\begin{equation*}
%	(A_i w )_\nu = \sum_{\rho=1}^n  w _\rho \frac{\partial f_\nu}{\partial  \xcrit_\rho(t-\tau_i)}\,.
%	\end{equation*}
%	Transposing and calculating the required derivative yields the matrix $\nabla_{ \xcrit} (w ^\prime A_i^\prime$) with components
%	%The derivative of this the transposed of this vector has the elements
%	\begin{equation*}
%	(\nabla_{ \xcrit} (w ^\prime A_i^\prime))_{\mu,\nu}=\sum_{\rho=1}^n  w_\rho\frac{\partial^2 f_\nu}{\partial \xcrit _\mu\,\partial  \xcrit_\rho(t-\tau_i)}\,
%	\end{equation*}
%	in its $\mu$-th row and $\nu$-th column. 
%	The other derivatives of vector matrix products can be found accordingly.
%}
The columns of $B$ span the normal space to the critical manifold in the space with the components $(x ,w, \alpha)$. Because we need to measure the distance to the critical manifold in the space of the parameters $\alpha$, we seek the particular normal direction that only has nonzero elements in the directions of the parameters $\alpha$. This corresponds to the linear combination $\kappa$ with 
\begin{equation*}
\begin{bmatrix}
	\nabla_{\xcrit}f^\prime&B_{12} & 0 \\
	0&B_{22}&2w \\
	\nabla_{ \alpha^{(c)}} f^\prime& B_{32}& 0 
      \end{bmatrix}\kappa = 
      \begin{bmatrix}
      0 \\ 0\\ r
      \end{bmatrix}\, , 
\end{equation*}
which is equivalent to~\eqref{eq:GeneralizedFoldNVSystem-NVSpan} and~\eqref{eq:GeneralizedFoldNVSystem-NVProjection}. These equations determine the normal direction $r$ up to its length. The last condition \eqref{eq:GeneralizedFoldNVSystem-Normalization} fixes the length of $r$. 
\end{proof}

The augmented system for a fold bifurcation manifold is the special case of \eqref{eq:ModfoldMani} with $\sigma = 0$, which
is the well-known augmented system (see, e.g., \cite{Kuznetsov1998})
\begin{subequations}\label{eq:FoldAug}
  \begin{align}
     f (\xcrit,...,\xcrit,\alpha^{(c)})&=0\label{eq:FoldSS}
    \\
    \sum_{k=0}^m A_k\,w&=0\label{eq:FoldEigVecDir}
    \\
    w^\prime w-1&=0\label{eq:FoldEigVecLength}\,.
  \end{align}
\end{subequations}
We can derive the normal vector system for the manifold defined by \eqref{eq:FoldAug} as a corollary to Prop.~\ref{prop:ModFoldNV}.  
\begin{corollary}[normal vector to manifold of fold bifurcations]
If  $(\xcrit$, $\alpha^{(c)}$, $w)$ is a regular solution to \eqref{eq:ModfoldMani} for $\sigma = 0 $ in the $2 n_x+1$ variables $x^{(x)}$, $w$ and one of the elements of $\alpha^{(c)}$, then $r$ that obeys the following equations is normal to the manifold of fold points at this solution:
\begin{subequations}\label{eq:FoldNVSystem}
\begin{align}
	\text{Equations~\eqref{eq:FoldAug}} \\
    \begin{bmatrix}
	\nabla_{\xcrit}f^\prime&-\sum_{k=0}^m \left[\nabla_{\xcrit} w^\prime A_k^\prime \right] & 0
	\\
	0&-\sum_{k=0}^mA_k^\prime &2w
%	\\
%	\nabla_{\Omega}f^\prime& -\sum_{i=0}^m\nabla_{\Omega}(w^\prime A_i^\prime  )& 0 
      \end{bmatrix}\kappa
      & = 0 
      %\label{eq:FoldNVSystem-NVSpan}
      \\
      \begin{bmatrix}
	\nabla_{\alpha}f^\prime& -\sum_{k=0}^m \left[\nabla_{\alpha^\text{(c)}} w^\prime A_k^\prime \right] &0\\
      \end{bmatrix}\kappa
      -r
      & = 0 
      %\label{eq:FoldNVSystem-NVProjection}
      \\
      r^\prime r- 1&=0
      %\label{eq:FoldNVSystem-Normalization}\,.
  \end{align}
\end{subequations}
\end{corollary}
The corollary follows directly from Prop.~\ref{prop:ModFoldNV} by substituting $\sigma=0$.

\subsubsection{Normal vector systems for Hopf and modified Hopf manifolds}\label{ssc:Hopf}
This section first briefly summarizes the augmented systems for Hopf and modified Hopf points as needed for the paper. The normal vector systems are derived subsequently. 

Assume $\xcrit$ is a steady state for parameter values $ \alpha^{(c)}$. 
If $\lambda=\sigma\pm\i\omega$ are the leading eigenvalues at this steady state, then there exists a $ w = a + \i b \in \mathbb{C}^n$  such that 
$x^{(c)}$, $\alpha^{(c)}$ and $w$  
obey the equations
\begin{subequations}\label{eq:GeneralizedHopfManifold}
\begin{align}
%-------------------------------------
 f( \xcrit, \xcrit,..., \xcrit,\alpha^{(c)})&=0\label{eq:GeneralizedHopfManifoldSteady}\\
%%---------------------------------------
\sigma a -\omega b -\sum_{k=0}^m A_k \big(a\, \expcos(\sigma,\omega,\tau_k) +b\, \expsin(\sigma,\omega,\tau_k) \big)&=0\label{eq:GeneralizedHopfManifoldReal}\\
%%-----------------------------------------
\sigma b + \omega a -\sum_{k=0}^m A_k\big(b\, \expcos(\sigma,\omega,\tau_k) -a\, \expsin(\sigma,\omega,\tau_k) \big)&=0\label{eq:GeneralizedHopfManifoldImag}\\
%-------------------------------------
 a ^\prime a + b ^\prime b   -1&=0\label{eq:GeneralizedHopfManifoldLength}\\
%-------------------------------------
 a^\prime b &=0\label{eq:GeneralizedHopfManifoldPhase}\,,
\end{align}
\end{subequations}
which are a simple extension of the augmented system for Hopf bifurcations of delay differential equations (see, e.g., \cite{Engelborghs2002}).
  Equations \eqref{eq:GeneralizedHopfManifoldSteady} hold, because  $x^{(c)}$ is a steady state by assumption. 
  Since $\lambda = \sigma \pm \mathrm{i}\omega $ are eigenvalues at this steady state by assumption, \eqref{eq:det} holds. Consequently, there exists a  $w = a + \i b \in \mathbb{C}^n$ such that 
\begin{equation*}
  \lambda  w -\sum_{k=0}^m \exp(-\lambda \tau_k)A_k w =0
\end{equation*}
and separating this equation into its real and imaginary part for $\lambda= \sigma+\i \omega$ yields~\eqref{eq:GeneralizedHopfManifoldReal} and \eqref{eq:GeneralizedHopfManifoldImag}, respectively. 
The eigenvector is normalized to unit length by \eqref{eq:GeneralizedHopfManifoldLength} and its complex phase is fixed by \eqref{eq:GeneralizedHopfManifoldPhase}.

The normal vector system can now be derived based on~\eqref{eq:GeneralizedHopfManifold}. 
\begin{proposition}[normal vector to manifold of modified Hopf points] \label{prop:ModHopfNV}
Let $\sigma\le 0$ be arbitrary but fixed. Let $\omega$, $a$ and $b$ be as in~\eqref{eq:GeneralizedHopfManifold}. 
Assume that  $(\xcrit, \alpha^{(c)},\omega, a, b)$ is a regular solution to \eqref{eq:GeneralizedHopfManifold} 
in the $3n_x+2$ variables $\xcrit, \omega, a, b $ and one of the elements of  $\alpha^{(c)}$. 
Then $r$ that obeys the following equations is normal to the manifold of modified Hopf points at this solution:
 \begin{subequations}\label{eq:GeneralizedHopfNV}
	\begin{align}
		\text{equations \eqref{eq:GeneralizedHopfManifold}}\label{eq:GeneralizedHopfNV1}
		\\
		\begin{bmatrix}
			\nabla_{ \xcrit} f^\prime& B_{12}& B_{13}&0&0
			\\
			0& B_{22}& B_{23}& 2 a & b \\
			0& B_{32}& B_{33}& 2 b & a \\
			0& B_{42}& B_{43}&   0 & 0 
		\end{bmatrix}\kappa&=0\label{eq:GeneralizedHopfNV2}
		\\
		\begin{bmatrix}
			\nabla_{ \alpha^{(c)}} f^\prime& B_{52}& B_{53}&0&0
		\end{bmatrix}\kappa- r&=0\label{eq:GeneralizedHopfNV3}
		\\
		 r^\prime  r-1 &=0\label{eq:NVNormingModHopf}
	\end{align} 
\end{subequations}
where
\begin{align*}
%%%%%%%%%%%%%%%%%%%%%%%%%%%%%%%%%%%%%%%%%%%%%%%%%%%%%%
	B_{12}=&\sum_{k=0}^m \Big(\Big. 
	\sigma 
	\big[\nabla_{ \xcrit}\tau_k\big]
	\big[\expcos(\sigma,\omega,\tau_k) a ^\prime+\expsin(\sigma,\omega,\tau_k) b ^\prime\big] 
	A_k^\prime
	-
	\omega
	\big[\nabla_{ \xcrit}\tau_k\big]
	\big[\expcos(\sigma,\omega,\tau_k) b ^\prime-\expsin(\sigma,\omega,\tau_k) a ^\prime\big]
	A_k^\prime
	\\
	&-
	\expcos(\sigma,\omega,\tau_k)
	\big[\nabla_{ \xcrit} a ^\prime A_k^\prime\big]
	-\expsin(\sigma,\omega,\tau_k)
	\big[\nabla_{ \xcrit} b ^\prime A_k^\prime\big]
	\Big.\Big)\,,
	\\
%%%%%%%%%%%%%%%%%%%%%%%%%%%%%%%%%%%%%%%%%%%%%%%%%%%%%%%
	B_{13}=&\sum_{k=0}^m  
	\Big(\Big.
	\sigma
	\big[\nabla_{ \xcrit}\tau_k\big]
	\big[\expcos(\sigma,\omega,\tau_k) b ^\prime-\expsin(\sigma,\omega,\tau_k) a ^\prime\big]
	A_k^\prime
	+\omega
	\big[\nabla_{ \xcrit}\tau_k\big] 
	\big[\expsin(\sigma,\omega,\tau_k) b ^\prime+\expcos(\sigma,\omega,\tau_k) a ^\prime\big]
	A_k^\prime
	\\
	&-\expcos(\sigma,\omega,\tau_k)
	\big[\nabla_{ \xcrit} b ^\prime A_k^\prime\big]
	+\expsin(\sigma,\omega,\tau_k)
	\big[\nabla_{ \xcrit} a ^\prime A_k^\prime\big]
	\Big.\Big)\,,
\end{align*}
\begin{align*}
	%%%%%%%%%%%%%%%%%%%%%%%%%%%%%%%%%%%%%%%%%%%%%%%%%%%%%%
	 B_{22}&=\sigma I-\sum_{k=0}^m \expcos(\sigma,\omega,\tau_k)A_k^\prime\,,
	 \quad
	%%%%%%%%%%%%%%%%%%%%%%%%%%%%%%%%%%%%%%%%%%%%%%%%%%%%%%
	 B_{23}=\omega I+\sum_{k=0}^m \expsin(\sigma,\omega,\tau_k)A_k^\prime\,,
	 \\
	 B_{32}&=-\omega I-\sum_{k=0}^m \expsin(\sigma,\omega,\tau_k)A_k^\prime\,,
	 \quad
	%%%%%%%%%%%%%%%%%%%%%%%%%%%%%%%%%%%%%%%%%%%%%%%%%%%%
	 B_{33}=\sigma I-\sum_{k=0}^m \expcos(\sigma,\omega,\tau_k)A_k^\prime\,,
\end{align*}
\begin{align*}
 B_{42}=&- b ^\prime+\sum_{k=0}^m\tau_k
 \big[\expsin(\sigma,\omega,\tau_k) a ^\prime-\expcos(\sigma,\omega,\tau_k) b ^\prime)\big]
 A_k^\prime\,,
 \quad
%%%%%%%%%%%%%%%%%%%%%%%%%%%%%%%%%%%%%%%%%%%%
 B_{43}=& a ^\prime+\sum_{k=0}^m \tau_k
 \big[\expsin(\sigma,\omega,\tau_k) b ^\prime+\expcos(\sigma,\omega,\tau_k) a ^\prime)\big]
 A_k^\prime\,,
\end{align*}
\begin{align*}
	B_{52}=
	&\sum_{k=0}^m 
	\Big(\Big.
	\sigma
	\big[\nabla_{ \alpha^{(c)}}\tau_k\big]
	\big[\expcos(\sigma,\omega,\tau_k) a ^\prime
	+\expsin(\sigma,\omega,\tau_k) b ^\prime\big]
	A_k^\prime
	+
	\omega
	\big[\nabla_{ \alpha^{(c)}}\tau_k\big]
	\big[\expsin(\sigma,\omega,\tau_k) a ^\prime
	-\expcos(\sigma,\omega,\tau_k) b ^\prime\big]
	A_k^\prime
	\\
	&- 
	\expcos(\sigma,\omega,\tau_k)
	\big[\nabla_{ \alpha^{(c)}}a^\prime A_k^\prime\big]
	-\expsin(\sigma,\omega,\tau_k)
	\big[\nabla_{ \alpha^{(c)}}b^\prime A_k^\prime\big]
	\Big.\Big)\,,
	\\
	%%%%%%%%%%%%%%%%%%%%%%%%%%%%%%%%%%%%%%%%%%%%%%%
	B_{53}=
	&\sum_{k=0}^m
	\Big(\Big. 
	\sigma
	\big[\nabla_{ \alpha^{(c)}}\tau_k\big]
	\big[\expcos(\sigma,\omega,\tau_k) b ^\prime
	-\expsin(\sigma,\omega,\tau_k) a ^\prime\big]
	A_k^\prime
	+
	\omega
	\big[\nabla_{ \alpha^{(c)}}\tau_k\big]
	\big[\expsin(\sigma,\omega,\tau_k) b ^\prime
	+\expcos(\sigma,\omega,\tau_k) a ^\prime\big]
	A_k^\prime
	\\
	&-
	\expcos(\sigma,\omega,\tau_k)
	\big[\nabla_{ \alpha^{(c)}} b ^\prime A_k^\prime\big]
	+\expsin(\sigma,\omega,\tau_k)
	\big[\nabla_{ \alpha^{(c)}} a ^\prime A_k^\prime\big]
	\Big.\Big)
	\,.
\end{align*}
%The expressions $\nabla_{ \xcrit} (a ^\prime A_i^\prime)$ are given by
%\begin{equation*}
%	(\nabla_{ \xcrit} (a ^\prime A_i^\prime))_{\mu,\nu}=\sum_{\rho=1}^n  a_\rho
%	\left.\frac{\partial^2 f_\nu(x(t),x(t-\tau_1),\dots,x(t-\tau_m),\Omega,\alpha)}
%	{\partial x _\mu\,\partial  x_\rho(t-\tau_i)}\right|_{x^{(c)}, \dots, x^{(c)}, \alpha^{(c)}}\,.
%\end{equation*}
%The expressions $\nabla_{ \xcrit} (b ^\prime A_i^\prime)$, $\nabla_{ \alpha^{(c)}} (a ^\prime A_i^\prime)$, $\nabla_{ \alpha^{(c)}} (b ^\prime A_i^\prime)$ are defined accordingly.
\end{proposition}

\begin{proof}
Equations \eqref{eq:GeneralizedHopfManifold} comprise $3n+2$ equations that depend 
on $3n+n_\alpha+1$ variables and parameters $ \xcrit$, $a$, $b$, $\omega$ and $ \alpha^{(c)}$. 
In the neighborhood of the given regular solution, 
these equations define an $(n_\alpha-1)$-dimensional manifold of modified Hopf points.
By the same arguments as in the proof to Prop.~\ref{prop:ModFoldNV}, the columns of the matrix $B$ that results from the outer product
\begin{align}\label{eq:BHopfHelper}
	B
	=&\begin{bmatrix}
		\nabla_{ \xcrit}\\
		\nabla_{ a }\\
		\nabla_{ b }\\
		\nabla_{\omega}\\
		\nabla_{ \alpha^{(c)}}\\
\end{bmatrix}
\begin{bmatrix}
%-------------------------------------
 f( \xcrit, \xcrit,..., \xcrit, \alpha^{(c)})\\
%---------------------------------------
\sigma a -\omega b -\sum_{k=0}^m A_k\big(\expcos(\sigma,\omega,\tau_k) a +\expsin(\sigma,\omega,\tau_k) b \big)\\
%-----------------------------------------
\sigma b+ \omega a  -\sum_{k=0}^m A_k\big(\expcos(\sigma,\omega,\tau_k) b -\expsin(\sigma,\omega,\tau_k) a \big)\\
%-------------------------------------
 a ^\prime a + b ^\prime b  -1\\
%-------------------------------------
 a ^\prime b 
\end{bmatrix}^\prime
\end{align} 
span the normal space to this manifold.
Evaluating~\eqref{eq:BHopfHelper} results in the matrix shown in~\eqref{eq:GeneralizedHopfNV2}, which can be seen as follows. 
Let $B$ be the block matrix with blocks $B_{jk}$, $j= 1, \dots, 4$, $k= 1, \dots, 5$  as in~\eqref{eq:GeneralizedHopfNV2} 
and blocks $B_{jk}$, $j= 5$, $k=1, \dots, 5$ as in~\eqref{eq:GeneralizedHopfNV3}. 
Block matrices $B_{14}$, $B_{15}$, $B_{21}$, $B_{31}$, $B_{41}$, $B_{44}$ and $B_{45}$ are zero, 
because the respective functions in~\eqref{eq:BHopfHelper} do not depend on the variables with respect to which the derivatives are calculated. 
The expressions for the block matrices $B_{11}$, $B_{24}$, $B_{34}$, $B_{25}$, $B_{35}$ are also obvious.  
$B_{22}$ results from the outer product
\begin{align*}
	B_{22}=&\nabla_{ a }\left(\sigma a ^\prime-\omega b ^\prime
	-\sum_{k=0}^m \big(\expcos(\sigma,\omega,\tau_k) a ^\prime-\expsin(\sigma,\omega,\tau_k) b ^\prime\big)A_k^\prime\right)
	=\sigma I-\sum_{k=0}^m \expcos(\sigma,\omega,\tau_k)A_k^\prime
\end{align*}
and $B_{23}$, $B_{32}$, $B_{33}$, $B_{42}$ and $B_{43}$ can be obtained from similar calculations. 
The remaining blocks, 
which contain the derivatives of \eqref{eq:GeneralizedHopfManifoldReal} and of~\eqref{eq:GeneralizedHopfManifoldImag} with respect to $x^{(c)}$ and $\alpha^{(c)}$,
require application of the chain and product rule, 
because the delays $\tau_k$ and the Jacobians $A_k$ are functions of $x^{(c)}$ and $\alpha^{(c)}$. 
This yields
\begin{align*}
	B_{12}
	=&
	\nabla_{ \xcrit} \left(
		\sigma a ^\prime-\omega b ^\prime-\sum_{k=0}^m \big[
			\expcos(\sigma,\omega,\tau_k) a ^\prime+\expsin(\sigma,\omega,\tau_k) b ^\prime
		\big]
		A_k^\prime 
	\right)
	 \\
	 &= -\sum_{k=0}^m\Big(
	 	\left[\nabla_{\xcrit} c_k\right]a^\prime A_k^\prime
	 	+\left[\nabla_{\xcrit} s_k\right]b^\prime A_k^\prime
	 	+c_k\left[\nabla_{\xcrit} a^\prime A_k^\prime\right]
	 	+s_k\left[\nabla_{\xcrit} b^\prime A_k^\prime\right]
	 \Big)
\end{align*}
Substituting 
\begin{align*}
  \nabla_{\xcrit}\expcos(\sigma, \omega, \tau_k)&= -\left(\sigma\expcos(\sigma, \omega, \tau_k)+ \omega\expsin(\sigma, \omega, \tau_k\right)\left[\nabla_{\xcrit}\tau_k\right]
  \mbox{ and}
  \\
  \nabla_{\xcrit}\expsin(\sigma, \omega, \tau_k)&= \left(-\sigma\expsin(\sigma, \omega, \tau_k)+\omega\expcos(\sigma, \omega, \tau_k)\right)\left[\nabla_{\xcrit}\tau_k\right]
\end{align*}
results in the expression for $B_{12}$ stated in the proposition.
The derivatives of \eqref{eq:GeneralizedHopfManifoldReal} with respect to $\alpha^{(c)}$, 
and of~\eqref{eq:GeneralizedHopfManifoldImag} with respect to $x^{(c)}$ and $\alpha^{(c)}$,
can be calculated accordingly and result in the expressions stated in the proposition for $B_{13}$
and, respectively, $B_{52}$ and $B_{53}$. 
Now recall the columns of $B$ span the normal space to the critical manifold in the space with the components $(x, a, b, \omega, \alpha)$. 
For the same reasons as stated in the proof of Prop.~\ref{prop:ModFoldNV}, we seek the particular linear combination $\kappa$ of the columns of $B$
that only has nonzero contributions in the subspace of the parameters $\alpha$. 
This corresponds to the linear combination $\kappa$ that solves~\eqref{eq:GeneralizedHopfNV2} and~\eqref{eq:GeneralizedHopfNV3}. 
If $\kappa$ solves these equations $\rho\kappa$ solves the equations for any $\rho\in\mathbb{R}$, $\rho\ne 0$. 
Because unit length of $r$ is enforced with~\eqref{eq:NVNormingModHopf}, $\kappa$ and $r$ are uniquely defined. 
\end{proof} 

The augmented system~\eqref{eq:HopfAug} and Prop.~\ref{prop:ModHopfNV} treat the case of modified Hopf points with leading eigenvalues $\lambda= \sigma\pm\i\omega$, $\sigma\le 0$. The case of Hopf bifurcations, i.e., $\sigma=0$, is obviously included. The augmented system for Hopf bifurcations of the treated class of delay differential equations reads (see, e.g., \cite{Engelborghs1999,Engelborghs2002})
\begin{subequations}\label{eq:HopfAug}
	\begin{align}
		%-------------------------------------
		f(\xcrit,\xcrit,...,\xcrit,\alphacrit)&=0\label{eq:HopfManifoldSteady}\\
		%---------------------------------------
		-\omega b-\sum_{k=0}^m A_k\big(\cos(\omega\tau_k) a+\sin(\omega\tau_k) b\big)&=0\label{eq:HopfManifoldReal}\\
		%-----------------------------------------
		\omega a-\sum_{k=0}^m A_k\big(\cos(\omega\tau_k) b-\sin(\omega\tau_k) a\big)&=0\label{eq:HopfManifoldImag}\\
		%-------------------------------------
		a^\prime a+b^\prime b -1&=0\label{eq:HopfManifoldLength}\\
		%-------------------------------------
		a^\prime b&=0 ,
	\end{align}
\end{subequations}
where $a$, $b$ and $\omega$ are as in~\eqref{eq:GeneralizedHopfManifold}. 
We state the normal vector system for completeness in the following corollary, which immediately follows from Prop.~\ref{prop:ModHopfNV} with $\sigma= 0$. 

\begin{corollary}[normal vector to manifold of Hopf bifurcations] \label{cor:HopfNV}
Assume  $(\xcrit$, $\alpha^{(c)}$, $\omega$, $a$, $b)$ is a regular solution to~\eqref{eq:HopfAug} in the variables $x^{(c)}$, $\omega$, $a$, $b$ and one of the elements of $\alpha^{(c)}$. 
Then $r$ that obeys \eqref{eq:HopfAug}, \eqref{eq:GeneralizedHopfNV2}, \eqref{eq:GeneralizedHopfNV3} and \eqref{eq:NVNormingModHopf} 
is normal to the Hopf bifurcation manifold at this solution, where the block matrices $B_{jk}$ from Prop.~\ref{prop:ModHopfNV} have to be replaced by 
\begin{align*}
	%%%%%%%%%%%%%%%%%%%%%%%%%%%%%%%%%%%%%%%%%%%%%%%%%%%%%%
	B_{12}=&
	-\sum_{k=0}^m 
	\Big(\Big.
		\omega
		\big[\nabla_{ \xcrit}\tau_k\big]
		\big[
			\cos(\omega\tau_k) b ^\prime-\sin(\omega\tau_k) a ^\prime
		\big]
		A_k^\prime
		-
		\cos(\omega\tau_k)
		\big[
			\nabla_{ \xcrit} a ^\prime A_k^\prime
		\big]
		-\sin(\omega\tau_k)
		\big[\nabla_{ \xcrit} b ^\prime A_k^\prime]
	\Big.\Big)	
	\,,
	\\
	%%%%%%%%%%%%%%%%%%%%%%%%%%%%%%%%%%%%%%%%%%%%%%%%%%%%%%%
	B_{13}=
	&\sum_{k=0}^m 
	\Big(\Big.
		\omega
		\big[\nabla_{ \xcrit}\tau_k\big] 
		\big[
			\sin(\omega\tau_k) b^\prime
			+\cos(\omega\tau_k) a ^\prime
		\big]
		A_k^\prime
		-
		\cos(\omega\tau_k)
		(\nabla_{ \xcrit} b ^\prime A_k^\prime)
		+\sin(\omega\tau_k)
		(\nabla_{ \xcrit} a ^\prime A_k^\prime)
	\Big.\Big)
	\,,
\end{align*}
\begin{align*}
	%%%%%%%%%%%%%%%%%%%%%%%%%%%%%%%%%%%%%%%%%%%%%%%%%%%%%%
	B_{22}=&-\sum_{k=0}^m \cos(\omega\tau_k)A_k^\prime\,,\quad
	%%%%%%%%%%%%%%%%%%%%%%%%%%%%%%%%%%%%%%%%%%%%%%%%%%%%%%
	B_{23}=\omega I+\sum_{k=0}^m \sin(\omega\tau_k)A_k^\prime\,,
	\\
	B_{32}=&-\omega I-\sum_{k=0}^m \sin(\omega\tau_k)A_k^\prime\quad
	%%%%%%%%%%%%%%%%%%%%%%%%%%%%%%%%%%%%%%%%%%%%%%%%%%%%
	B_{33}=-\sum_{k=0}^m \cos(\omega\tau_k)A_k^\prime\,,
\end{align*}
\begin{align*}
	B_{42}=&- b ^\prime+\sum_{k=0}^m\tau_k\big[\sin(\omega\tau_k) a ^\prime-\cos(\omega\tau_k) b ^\prime)\big]A_k^\prime\,,
	\quad
	B_{43}=& a ^\prime+\sum_{k=0}^m \tau_k\big[\sin(\omega\tau_k) b ^\prime+\cos(\omega\tau_k) a ^\prime)\big]A_k^\prime\,,
\end{align*}
\begin{align*}
	B_{52}=
	&\sum_{k=0}^m 
	\Big(\Big.
		\omega
		\big[\nabla_{ \alpha^{(c)}}\tau_k\big]
		\big[\sin(\omega\tau_k) a ^\prime-\cos(\omega\tau_k) b ^\prime\big]
		A_k^\prime
		- 
		\cos(\omega\tau_k)
		\big[\nabla_{ \alpha^{(c)}}a^\prime A_k^\prime\big]
		-\sin(\omega\tau_k)
		\big[\nabla_{ \alpha^{(c)}}b^\prime A_k^\prime\big]
	\Big.\Big)
	\,,
	\\
	%%%%%%%%%%%%%%%%%%%%%%%%%%%%%%%%%%%%%%%%%%%%%%%
	B_{53}=
	&\sum_{k=0}^m 
	\Big(\Big.
		\omega
		\big[\nabla_{ \alpha^{(c)}}\tau_k\big]
		\big[\sin(\omega\tau_k) b ^\prime+\cos(\omega\tau_k) a ^\prime\big]
		A_k^\prime
		- 
		\cos(\omega\tau_k)
		\big[\nabla_{ \alpha^{(c)}} b ^\prime A_k^\prime\big]
		+\sin(\omega\tau_k)
		\big[\nabla_{ \alpha^{(c)}} a ^\prime A_k^\prime\big]
	\Big.\Big)
	\,.
\end{align*}
\end{corollary}

\section{Laser diode optimization with normal vector constraints}\label{sec:LaserNVOptim}
With the optimization goal stated in Section~\ref{sec:Model} and the constraints for robust stability described in the previous section, we are able to state the full optimization problem with constraints for robustness. It reads
\begin{subequations}\label{eq:optimProbComplete}
\begin{align}
\min~-|A|^2&\label{eq:optimProbCostComplete}\\
\text{s.\,t.}\quad0 &= f(x^{(0)},x^{(0)},\alpha^{(0)})\label{eq:optimProbStStComplete}\\
0.02 + \Delta\tau\cdot\sqrt{n_\alpha} <& j^{(0)} < 1.0265 -\Delta\tau\cdot\sqrt{n_\alpha} \label{eq:optimProbCostConstraintj}\\
0.0001+\Delta\eta\cdot\sqrt{n_\alpha}  <&\eta^{(0)} < 0.02-\Delta\eta\cdot\sqrt{n_\alpha} \label{eq:optimProbCostConstrainteta}\\
220 +\Delta\tau\cdot\sqrt{n_\alpha} <&\tau^{(0)} < 2000-\Delta\tau\cdot\sqrt{n_\alpha}\label{eq:optimProbCostConstrainttau} \\
0&=G(\xcrit,\alphacrit,u^{(c)} )\label{eq:optimProbCompleteMani}\\
0&=H(\xcrit,\alphacrit,u^{(c)} ,\kappa,r)\label{eq:optimProbCompleteNVSys}\\
0&=r^\prime r-1\label{eq:optimProbCompleteNVnorm}\\
0&=\alpha^{(0)}-\alpha^{(c)} + d\frac{r}{||r||}\label{eq:optimProbCompleteConnection}\\
d&>\sqrt{n_\alpha}\Delta \alpha\label{eq:optimProbCompleteDistance}\,,
\end{align}
\end{subequations}
where $x= [\real{A}, \imag{A}, N_1, \dots, N_K, \Omega]^\prime$, $\alpha= [j, \eta, \alpha_{LW}, \phi, \tau, T, d]^\prime$,
and the superscripts $(c)$ and $(0)$ refer to the critical and nominal points, respectively,
%and the degrees of freedom are 
%$x^{(0)}$, $j^{(0)}$, $\eta^{(0)}$, $\tau^{(0)}$, $x^{(c)}$, $\alpha^{(c)}$, $u^{(c)}$, $\kappa$, $r$ and $d$. 
%
Equation~\eqref{eq:optimProbStStComplete} enforces steady state operation. 
Constraints~\eqref{eq:optimProbCostConstraintj}--\eqref{eq:optimProbCostConstrainttau} 
are the robust counterparts to~\eqref{eq:NonrobustConstraintj}--\eqref{eq:NonrobustConstrainttau}. 
Constraints~\eqref{eq:optimProbCompleteMani} and~\eqref{eq:optimProbCompleteNVSys} are the normal vector constraints. 
Since a robust distance to a manifold of generalized Hopf points has to be enforced in this particular example, they are of the 
type~\eqref{eq:GeneralizedHopfManifold} and~\eqref{eq:GeneralizedHopfNV2}--\eqref{eq:NVNormingModHopf}, respectively. 
Just as in the motivating example in Section~\ref{sec:LaserProblemOutline}, there exist $n_\alpha= 7$ uncertain parameters (see Table~\ref{tab:Params}). 
Four of these uncertain parameters are assumed to be uncertain but fixed, whereas the remaining three,  $j$, $\eta$ and $\tau$, can be altered to maximize the laser intensity.

While the optimization \eqref{eq:optimProb} resulted in an unstable laser operation, for which small disturbances triggered quasi-chaotic behavior, robust stability is guaranteed in~\eqref{eq:optimProbComplete}.
The robustness constraints \eqref{eq:optimProbCompleteMani}-\eqref{eq:optimProbCompleteDistance} prevent crossing the Hopf bifurcation manifolds shown in Figure~\ref{fig:3DParam31}. The parameter space shown in this figure is spanned by the optimization variables. It is evident from the figure that Hopf bifurcations always appear as the feedback strength $\eta$ is increased. 
The 3-dimensional plot in Figure~\ref{fig:3DParam31} also underlines the necessity to interpret 2-dimensional visualizations with care, especially as the the space of uncertain variables is 7-dimensional in this example.

\begin{figure}[hbtp]
	\center
	\includegraphics[width=0.7\textwidth]{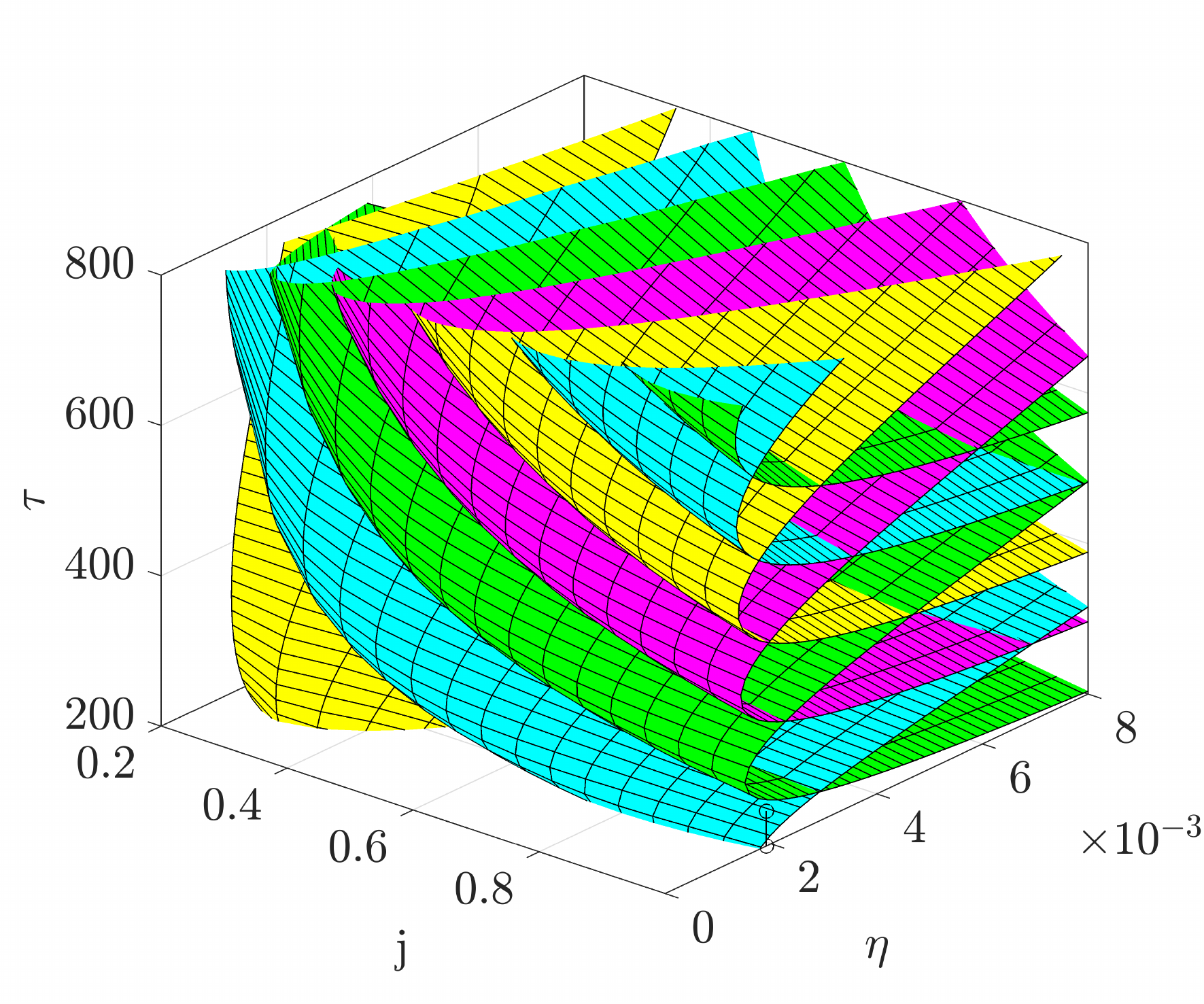}
	\caption{Hopf bifurcation manifolds in the $j$-$\eta$-$\tau$-space. The manifolds are colored differently for easy distinction. \label{fig:3DParam31}}
\end{figure}

The normal vector constraints \eqref{eq:optimProbCompleteMani}-\eqref{eq:optimProbCompleteDistance} take only one critical manifold into account, while multiple such manifolds obviously exist (see Figure~\ref{fig:3DParam31}). % fig:3DParam31
Multiple manifolds can be treated rigorously with an approach that resembles active set optimization methods~\cite{Monnigmann2007}. We do not discuss this aspect here in detail, however, since only one Hopf manifold plays a role.

Solving~\eqref{eq:optimProbComplete} results in the optimal parameters  $j^{(0)}=1$, $\eta^{(0)} = 1.911\cdot10^{-3}$ and $\tau^{(0)} = 246.46$. The laser diode operates at an intensity of $|A|^2=6.268$ at the optimal robust steady state. 
We checked the stability of the steady state solutions at the $2^{n_\alpha}=128$ vertices of the robustness hypercube~\eqref{eq:uncertaintyHyperrectangle} to corroborate the optimization result. 
If compared to the previous unstable result, the feedback strength $\eta$ is smaller. Its upper bound was active in the previous optimization without robustness constraints. In the current robustly stable optimization, it is bounded by the stability requirement. In both cases, the optimal pump current $j$  and the feedback delay $\tau$ are bounded by their upper and lower bounds, respectively. However, the bounds on $j$ and $\tau$ have to be tightened in the current optimization to guarantee that the bounds are met even in the presence of uncertainty (cf. Figure~\ref{fig:Param31zoom}-\ref{fig:Param312zoom}).
These parameter differences result in different intensities $|A|^2$. The restrictions leading to a robustly stable laser operation result in an intensity reduction of about $6\%$. This reduction can be interpreted to be the price for stability and robustness.

Figures~\ref{fig:Param31}--\ref{fig:ParamAlphaPhi31} visualize the robustly stable optimal point and its robustness regions in the parameter space. All figures contain cuts of the manifolds shown in Figure~\ref{fig:3DParam31}.

Figure~\ref{fig:Param31} illustrates the restricting nature of the stability constraints in the $\eta$-$\tau$-plane.
The value of $\eta$ cannot be increased further without violating the robustness constraint. Parts of the hyperspherical outer approximation of the uncertainty region appear to lie in the unstable region. We stress, however, the hypersphere remains completely in the stable region. 
The impression conveyed by Figure~\ref{fig:Param31} is merely caused by the 2-dimensional cut of the 7-dimensional parameter space.

Figure~\ref{fig:Param312} shows the optimal point in the $j$-$\tau$-plane. 
It is evident from Figure~\ref{fig:Param312}b that, in addition to the lower bound from~\eqref{eq:optimProbCostConstrainttau}
on $\tau$, the upper bound from~\eqref{eq:optimProbCostConstraintj} 
on $j$ is active at the optimal point.
  
\begin{figure*}[hbtp]
	\center
	\subfloat[View of the optimal solution on about the same scale as used in Figure~\ref{fig:3DParam31}. The interior of the dashed rectangle is shown in detail in Figure~\ref{fig:Param31zoom}. \label{fig:Param31all}]{\includegraphics[width=0.48\textwidth]{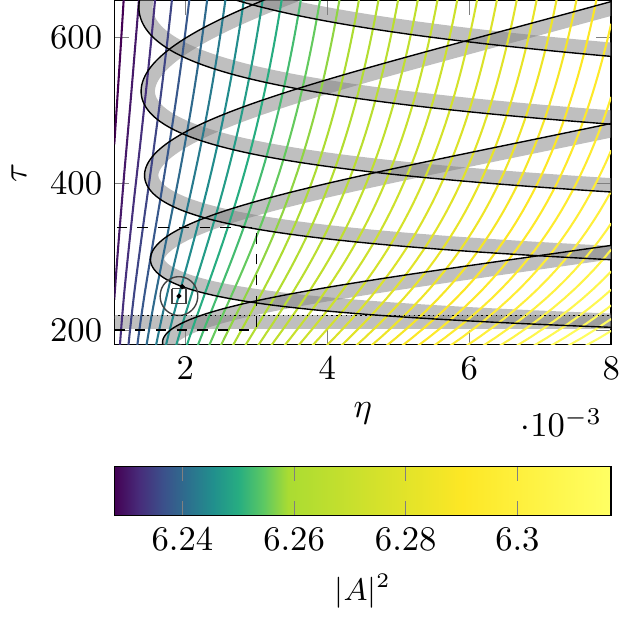}}\quad
	\subfloat[More detailed view of the nominal point, uncertainty region and active constraints.\label{fig:Param31zoom}]{\includegraphics[width=0.48\textwidth]{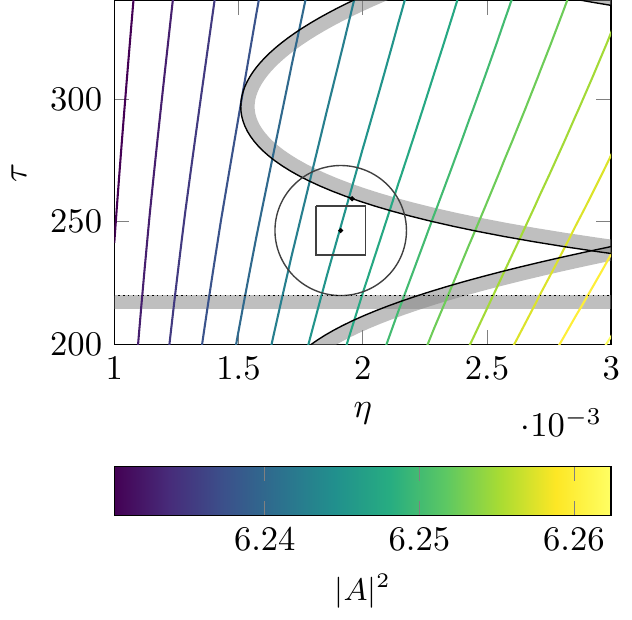}}
	\caption{Optimal point that results from solving~\eqref{eq:optimProbComplete} in the $\eta$-$\tau$-plane.
	Solid black lines are manifolds of Hopf bifurcation points, where the unstable sides are marked by grey bands. The botted black line marks the lower bound from~\eqref{eq:optimProbCostConstrainttau}. \label{fig:Param31}}
\end{figure*}

\begin{figure*}[hbtp]
	\center
	\subfloat[View of the optimal solution on about the same scale as in Figure~\ref{fig:3DParam31}. The interior of the dashed rectangle is shown in detail in Figure~\ref{fig:Param312zoom} \label{fig:Param312all}]{\includegraphics[width=0.48\textwidth]{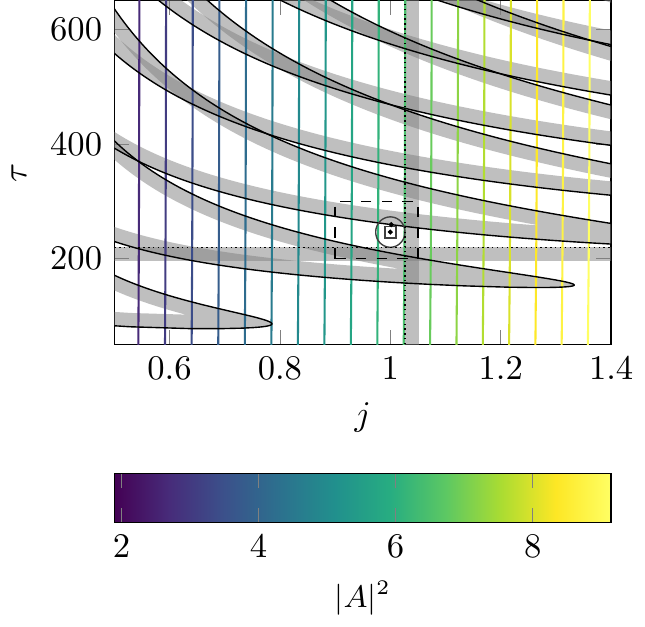}}\quad
	\subfloat[\label{fig:Param312zoom} More detailed view of nominal point, uncertainty region and active constraints.]{\includegraphics[width=0.48\textwidth]{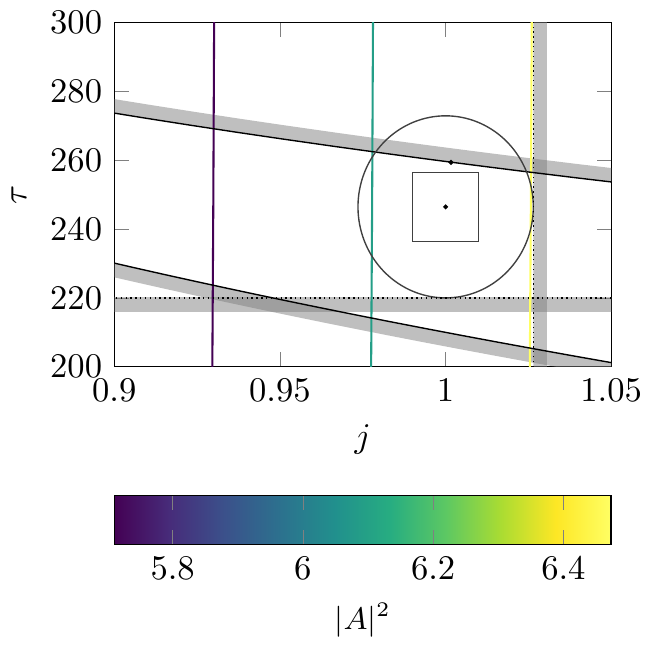}}
	\caption{Optimal point that results from solving~\eqref{eq:optimProbComplete} in the $j$-$\eta$-plane. 
	Dotted black lines are the lower bound and upper bound of $\tau$ and $j$. All other elements have the same meaning as in Figure~\ref{fig:Param31}. \label{fig:Param312}}
\end{figure*}

Figure~\ref{fig:ParamAlphaPhi31} illustrates the optimal point in the $\phi$-$\alpha_{LW}$-plane.  This plane is chosen, because the components of the normal vector along the $\phi$- and $\alpha_{LW}$-axes are the largest ones. The absence of contour lines in this figure accounts for the fact, that $\alpha_{LW}$ and $\phi$ are assumed to be uncertain, but cannot be modified for optimization. %The same holds for $d$ and $T$, but these parameters' impact on optimum and stability is rather small.

\begin{figure}[hbtp]
	\center
	\includegraphics[]{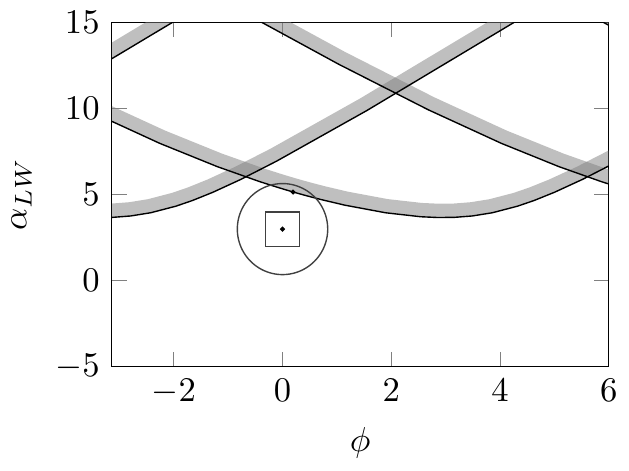}
	\caption{Optimal point plotted at $\alpha_{LW}$-$\phi$-plane. What appear to be multiple manifolds, is in fact only a single manifold repeating at $2\pi$ intervals. \label{fig:ParamAlphaPhi31}}
\end{figure}

The dynamical behavior of the system at the optimal point is illustrated with Figure~\ref{fig:simStable} for $K=31$. The laser diode converges to its steady state, despite the initial disturbance.
\begin{figure}[hbtp]
\center
\includegraphics[]{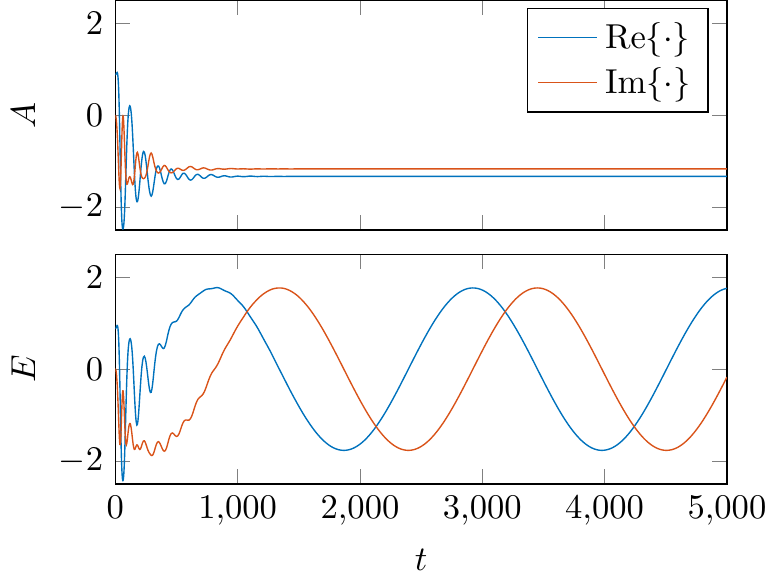}
\caption{Simulation of laser diode \eqref{eq:ODECarrierModelRot} for $K=31$ with optimal parameters. The initial steady state is disturbed at $t=0$, the electrical field is reduced by $50\,\%$. At $t=0.1$, the laser diode has already returned to its steady state. Results are shown both in rotating coordinates $A$ and fixed coordinates $E$. 
\label{fig:simStable}}
\end{figure}

\section{Conclusion}\label{sec:Conclusion}
We introduced a method for the robust optimization of parametrically uncertain finite-dimensional nonlinear dynamical systems with state- and parameter-dependent delays.
Stability and robustness are guaranteed with constraints that enforce a lower bound on the distance of the optimal point to submanifolds of saddle-node and Hopf bifurcations on the steady state manifold. The lower bound on the distance can be interpreted as a finite variation of the model parameters under which stability must be guaranteed; it therefore is a useful robustness measure in applications.  
We illustrated the proposed method with the optimization of a laser diode system with a delay due to an external cavity. 

\section*{Acknowledgements}
Funding by Deutsche Forschungsgemeinschaft (grant number MO 1086/13-1) is gratefully acknowledged.

\bibliographystyle{plain}
\bibliography{References} 

%\appendix 
%
%\section{Proofs}
%\CommentMmo{Rename section to 'proof of ...'.}
%\begin{proof}
%	Assume $\tilde{a}$, $\tilde{b}\in\R^n$ exist that obey \eqref{eq:GeneralizedHopfManifoldSteady} - \eqref{eq:GeneralizedHopfManifoldAuxCond}, except \eqref{eq:GeneralizedHopfManifoldPhase}. Define $a$, $b$ by $a+\i b= e^{\i\phi}(\tilde{a}+ \i \tilde{b})$, where $\phi$ is arbitrary and still has to be defined. Since $a+ \i b$ is an eigenvector if and only if $e^{\i\phi}(\tilde{a}+ \i \tilde{b})$ is an eigenvector, and since $|e^{\i\phi}|=1$, \eqref{eq:GeneralizedHopfManifoldSteady} - \eqref{eq:GeneralizedHopfManifoldLength} and \eqref{eq:GeneralizedHopfManifoldAuxCond} also hold for $a+ \i b$. Substituting $a+\i b= e^{\i\phi}(\tilde{a}+ \i \tilde{b})$ into \eqref{eq:GeneralizedHopfManifoldPhase} yields   
%	\begin{align}\label{eq:GeneralizedHopfManifoldPhaseRot}
%	0&=(\cos(\phi)\tilde{a}-\sin(\phi)\tilde{b})^\prime(\sin(\phi)\tilde{a}+\cos(\phi)\tilde{b})\nonumber\\
%	&=\cos(2\phi)\tilde{a}^\prime \tilde{b} +\frac{1}{2}\sin(2\phi)(\tilde{a}^\prime \tilde{a} -\tilde{b}^\prime \tilde{b})
%	\end{align}
%	Since $\tilde{a}^\prime\tilde{b}\ne 0$ by assumption, \eqref{eq:GeneralizedHopfManifoldPhaseRot} can be solved for $\phi$. 
%	Equation \eqref{eq:GeneralizedHopfManifoldPhase} is then fulfilled for the resulting $a+\i b$.
%\end{proof}

%\typeout{col width is \the\columnwidth} 
\end{document}